\documentclass[11pt]{amsart}
\usepackage[margin=1in]{geometry}

\usepackage{amssymb}
\usepackage{amsthm}
\usepackage{amsmath}
\usepackage{mathrsfs}
\usepackage{amsbsy}
\usepackage{bm}
\usepackage{hyperref}
\usepackage{tikz}
\usepackage{array}
\usepackage{enumerate}
\usepackage{bbm}
\usepackage{comment}
\usepackage{mathtools}
\usepackage{makecell} 
\usepackage{colortbl}
\usepackage{xcolor}

\definecolor{LightBlue}{rgb}{0,0.8,1} 

\hypersetup{colorlinks=true, citecolor=LightBlue, linkcolor=blue,urlcolor=blue} 

\usepackage{hhline}
\allowdisplaybreaks
\usepackage[noadjust]{cite}

\usepackage{caption}
\usepackage[noabbrev,capitalise,nameinlink]{cleveref}
\crefname{conjecture}{Conjecture}{Conjectures}

\newtheorem{theorem}{Theorem}[section]

\newtheorem{lemma}[theorem]{Lemma}

\theoremstyle{definition}

\newtheorem{remark}[theorem]{Remark}

\newcommand{\ba}{\mathbf{a}}
\newcommand{\bb}{\mathbf{b}}
\DeclareMathOperator{\id}{id}

\title{On the number of permutation-twisted dot products}
\author{Ruben Carpenter}
\address{Yale University, 421 Temple St, New Haven, CT 06511}
\email{ruben.carpenter@yale.edu}
\author{Colin Defant}
\address{Harvard University, 1 Oxford St, Cambridge, MA 02139}
\email{colindefant@gmail.com} 
\author{Noah Kravitz}
\address{St John's College, Oxford and Mathematical Institute, University of Oxford; St Giles', Oxford OX1 3JP, UK}
\email{noah.kravitz@maths.ox.ac.uk}

\begin{document}

\begin{abstract}
Let $\mathbb{K}$ be a field of characteristic $0$. For each choice of distinct $a_1, \ldots, a_n\in \mathbb{K}$ and distinct $b_1, \ldots, b_n\in \mathbb{K}$, consider the sum $S=\sum_{i=1}^n a_i b_{\pi(i)}$ as $\pi$ ranges over the permutations of $[n]$.  We show that this sum always assumes at least $\Omega(n^3)$ distinct values. This ``support'' bound, which is optimal up to the value of the implicit constant, complements recent work of Do, Nguyen, Phan, Tran, and Vu, and of Hunter, Pohoata, and Zhu on the anticoncentration properties of $S$ when $a_1,\ldots,a_n,b_1,\ldots,b_n$ are real and $\pi$ is chosen uniformly at random.
\end{abstract}

\maketitle

\section{Introduction}

Let $\ba=(a_1,\ldots,a_n)$ and $\bb=(b_1,\ldots,b_n)$ be $n$-tuples of elements of a field $\mathbb K$ of characteristic $0$.  For each permutation $\pi$ in the symmetric group $\mathfrak S_n$, define the quantity
$$S(\ba,\bb;\pi):=\sum_{i=1}^n a_i b_{\pi(i)},$$
which is the ``dot product of $\ba,\bb$ twisted by $\pi$''.  Motivated by Littlewood--Offord theory and a geometric problem of Pawlowski~\cite{P24}, Do, Nguyen, Phan, Tran, and Vu \cite{DNPTV25}, and independently Hunter, Pohoata, and Zhu \cite{HPZ26} recently raised the problem of studying the anticoncentration properties of the random variable $S(\ba,\bb;\pi)$ when $\ba$ and $\bb$ are fixed real vectors and $\pi$ is chosen uniformly at random.  One of their main results states that if $a_1,\ldots, a_n\in\mathbb R$ are distinct and $b_1,\ldots,b_n\in\mathbb R$ are distinct, then each possible value of the sum is assumed with probability at most $O(n^{-5/2}\log n)$; this bound is optimal up to the logarithm (which likely can be removed).  The paper \cite{AaronEtAl} has rounded out this small flurry of recent activity.

We are interested in the set
$$\mathcal{S}(\ba, \bb):=\left\{S(\ba,\bb; \pi) : \pi \in \mathfrak S_n\right\}$$
of all possible sums as $\pi$ ranges; this is of course precisely the support of the random variable $S(\ba,\bb;\pi)$.  Our main result states that if $a_1,\ldots,a_n$ are all distinct and $b_1,\ldots,b_n$ are all distinct, then there are at least $\Omega(n^3)$ distinct sums. 

\begin{theorem}\label{thm:main}
There exists an absolute constant $c > 0$ such that the following holds. For every field $\mathbb K$ of characteristic $0$, we have $|\mathcal{S}(\ba,\bb)| \geq c n^3$ for all $\ba, \bb\in\mathbb K^n$ each with distinct entries. 
\end{theorem}

Some version of the distinctness condition on $a_1,\ldots,a_n$ and $b_1,\ldots,b_n$ is necessary since otherwise $S(\ba,\bb; \pi)$ may not vary at all as $\pi$ varies.  In the case where $\ba,\bb$ have real coordinates, one can obtain the trivial lower bound $|\mathcal S(\ba,\bb)| \geq 1+\binom{n}{2}$ by first ordering both tuples $\ba,\bb$ in increasing order and then iteratively swapping a pair of adjacent elements of $\bb$ until all elements are in decreasing order.  The bound in \Cref{thm:main} is sharp up to the value of the constant $c$ since $|\mathcal{S}(\ba, \bb)|=(1/6+o(1))n^3$ for $\ba=\bb=(1,2,\ldots, n)$.

To prove \cref{thm:main}, we first handle the special case where $\ba$ and $\bb$ are real vectors, which is already of interest. Our proof is elementary and fairly short.  Pohoata \cite{P26} has recently obtained an independent proof of this special case of \cref{thm:main} using more complicated methods. 

We then prove \cref{thm:main} when $\mathbb K$ is the field $\mathbb C$ of complex numbers. To do so, we invoke Beck's theorem from discrete geometry in order to split into two cases. In the first case, each of $\ba,\bb$ contains a large collinear subset, and we can reduce the argument to the real setting. In the second case, one of $\ba,\bb$ defines many distinct lines, and we can run a different argument based on the $2$-dimensional case of the following theorem, which may be of independent interest. 
\begin{theorem}\label{thm:gamma} 
For each positive integer $d$, there exists a constant $\gamma_d>0$ such that the following holds. If $A \subseteq \mathbb{R}^d$ is an $m$-element set with no linearly dependent $d$-element subsets, then $A$ has at least $\gamma_dm^{d+1}$ distinct subset sums. 
\end{theorem}

We deduce the full version of \cref{thm:main} from the complex case using the fact that every finitely generated field extension of $\mathbb Q$ embeds into $\mathbb C$. 

\section{Real Numbers}\label{sec:reals} 
In this section, let $\ba,\bb \in \mathbb{R}^n$ be $n$-tuples each with distinct entries.
Write $S_0:=S(\ba,\bb; \id)$, where $\id$ denotes the identity permutation.  For the permutation $\pi=(j,k)$ that swaps $j,k$ and fixes all elements of $[n]\setminus\{j,k\}$, we can compute
$$S(\ba,\bb; (j,k))=S_0+a_j(b_k-b_j)+a_k(b_j-b_k)=S_0-(a_k-a_j)(b_k-b_j).$$
In general, if $\pi=(j_1,k_1) \cdots (j_r,k_r)$ is a product of disjoint transpositions, then we obtain
$$S(\ba,\bb;(j_1,k_1) \cdots (j_r,k_r))=S_0-\sum_{i=1}^r (a_{k_i}-a_{j_i})(b_{k_i}-b_{j_i}).$$
Thus, to prove \Cref{thm:main}, it suffices to find disjoint transpositions $(j_1,k_1), \ldots, (j_r,k_r)$  with the property that the set
\begin{equation}\label{eq:differences}
\{(a_{k_1}-a_{j_1})(b_{k_1}-b_{j_1}), \ldots, (a_{k_r}-a_{j_r})(b_{k_r}-b_{j_r})\}
\end{equation}
has $\Omega(n^3)$ distinct subset sums.  We are of course free to reorder the entries of $\ba$ and of $\bb$ (possibly in different ways) before specifying our family of disjoint transpositions.

In general, a set of $m$ real numbers can have as few as $O(m^2)$ distinct subset sums (one example is $\{1, \dots, m\}$). To get a cubic lower bound for the set \eqref{eq:differences}, we will choose transpositions that capture the behavior of $\ba$ and $\bb$ at many different scales.

\subsection{Constructing the involutions}\label{subsec:involutions} 

The first step is identifying a family of disjoint pairs $(j_i,k_i)$ of elements of $[n]$ such that the corresponding differences $a_{k_i}-a_{j_i}$ ``see'' the various scales of $\ba$, and likewise for $\bb$.  We achieve this iteratively, as follows.

Without loss of generality, we may assume that $a_1<\cdots<a_n$.  To begin, among all choices of $1 \leq j<k \leq n$, choose a pair minimizing the quantity $a_k-a_j$, and let $(j_1,k_1)$ denote this pair. Suppose we have already obtained disjoint pairs $(j_1, k_1), \ldots, (j_s, k_s)$  (with each $j_i<k_i$).  Let $\mathcal{U}_s$ be the set of $n-2s$ unpaired indices of $[n]$.  Among all choices of $j<k$ in $\mathcal{U}_s$ with
$$|\{j+1, \dots, k-1\}\cap \mathcal{U}_s| = s-1,$$ 
choose one that minimizes $a_k-a_j$, and let $(j_{s+1},k_{s+1})$ denote this pair; if there are no such pairs, then halt the procedure.

Suppose that the procedure produced the pairs $(j_1,k_1), \ldots, (j_r,k_r)$ and then halted.  Since it failed to produce $(j_{r+1},k_{r+1})$, we must have had $n-2r=|\mathcal{U}_r|\leq r+1$, i.e., $r\geq (n-1)/3$.  Thus, our process always produces a linear number of pairs.  For each $1 \leq i \leq r$, write $x(i):=a_{k_i}-a_{j_i}$.  The crucial property of the sequence $x(1), \ldots, x(r)$ (besides its length) is its superadditivity.

\begin{lemma}[Superadditivity]
    We have $x(u+v) \geq x(u) + x(v)$ for all $1\leq u, v\leq r$ with $u+v\leq r$.
\end{lemma}
\begin{proof}
    Write $i_{1} < i_2 < \cdots < i_{u+v-1}$ for the elements of $\mathcal{U}_{u+v}$ between $j_{u+v}$ and $k_{u+v}$. All of these elements are also in $\mathcal{U}_{u-1}$ and $\mathcal{U}_{v-1}$, so the greedy nature of our inductive procedure guarantees that $x(u) \leq a_{i_u} - a_{j_{u+v}}$ and $x(v) \leq a_{k_{u+v}} - a_{i_u}$, whence $x(u) + x(v) \leq a_{k_{u+v}} - a_{j_{u+v}} = x(u+v)$.
\end{proof}

In the same way we obtain a superadditive sequence $y(1), \ldots, y(r')$ from $\bb$, again with $r'=\Omega(n)$. It remains to show that the set of numbers $s(i) := x(i) y(i)$, for $i$ between $1$ and $\min(r, r') = \Theta(n)$, has $\Omega(n^3)$ distinct subset sums. 

\subsection{Subset sums for dot products of superadditive sequences}

For each $m\geq 1$, let $f(m)$ denote the minimum possible number of distinct subset sums of $\{s(1), \ldots, s(m)\}$, where $x(1),\ldots,x(m)$ and $y(1), \dots, y(m)$ are superadditive sequences of positive real numbers and $s(i):=x(i)y(i)$.  In light of the preceding discussion, the reader will be unsurprised to see that our main technical lemma takes the following form. 

\begin{lemma}\label{lem:key}
    There is a constant $\kappa>0$ such that $f(m)\geq \kappa m^3$ for all $m\geq 1$. 
\end{lemma}

\begin{proof}
Let $x(1),\ldots,x(m)$ and $y(1),\ldots,y(m)$ be superadditive sequences of positive reals, and set $s(i):=x(i)y(i)$. The key claim is that there exists $t(m)=(1- O(m^{-1/3}))\cdot (3m^2)^{1/3}$ such that 
\begin{equation*}\label{eq:t}
\sum_{i=1}^{t(m)} s(i) < s(m).
\end{equation*}
Before proving this claim, let us see how it implies the conclusion of the lemma.

All subset sums of $\{s(1), \dots, s(m-1)\}$ are also subset sums of $\{s(1), \dots, s(m)\}$. Moreover, for each $S\subseteq \{1, \dots, t(m)\}$, the subset sum
$$\sum_{i\in \{1, \dots, m\}\setminus S} s(i)= \sum_{i=1}^{m-1} s(i) - \sum_{j\in S} s(j) + s(m)$$ is
larger than $s_1 + \cdots + s_{m-1}$, so it is a subset sum of $\{s(1),\ldots,s(m)\}$ that is not a subset sum of $\{s(1),\ldots,s(m-1)\}$. By varying $S$, we obtain at least $f(t(m))$ new subset sums, so
\begin{equation}\label{eq:f_recurrence}
f(m)\geq f(m-1) + f(t(m)).
\end{equation} 

Let $\kappa_m:=m^{-3}f(m)$. We show that $\kappa_m$ is uniformly bounded away from $0$. By the recurrence~\eqref{eq:f_recurrence}, 
\begin{align*}
    f(m) &\geq f(m-1) + f(t(m)) \\
    &\geq \kappa_{m-1} (m-1)^3 + \kappa_{t(m)} \left((1-O(m^{-1/3})) \cdot (3m^2)^{1/3}\right)^3 \\
    &\geq \min\left\{\kappa_{m-1}, (1-O(m^{-1/3}))\cdot \kappa_{t(m)} \right\} \cdot ((m-1)^3 + 3m^2) \\
    &\geq \min\left\{\kappa_{m-1}, \exp\left(-O(m^{-1/3})\right)\cdot \kappa_{t(m)} \right\}\cdot m^3,
\end{align*}
so 
$$\kappa_{m}\geq \min\left\{\kappa_{m-1}, \exp\left(-O(m^{-1/3})\right)\cdot \kappa_{t(m)}\right\}.$$
Thus, even if $\kappa_1, \kappa_2, \kappa_3, \dots$ does decrease, it does so \textit{extremely} slowly.  Concretely, there is some natural number $m^*$ such that for any $m$, the sequence $m_0:=m, ~ m_1:=t(m_0), ~m_2:=t(m_1), ~\ldots$ eventually reaches some $m_q \leq m^*$.  In this case we can bound
$$\kappa_m \geq \prod_{i=0}^{q-1} \exp\left(-O(m_i^{-1/3})\right) \cdot \min_{1\leq m'\leq m^*} \kappa_{m'}\gg \exp\left(-\sum_{i=0}^{q-1} O(m_i^{-1/3})\right).$$
As $i$ decreases, the sequence $m_i$ grows doubly-exponentially, so the product on the right-hand side is uniformly bounded below by some $\kappa>0$.

It remains to prove the main claim.  The intuition is that for $m=\left(\sum_{i=1}^t i^2\right)^{1/2}\approx (t^3/3)^{1/2}$, two ``fractional'' applications of superadditivity morally give the chain of inequalities
$$\sum_{i=1}^{t} x(i)y(i) \leq \sum_{i=1}^t \frac{i}{m} x(m) \cdot y(i) \leq x(m) y\left(\sum_{i=1}^t \frac{i^2}{m}\right) = x(m)y(m).$$
To make this heuristic precise, we must keep our applications of superadditivity ``in the integers''.

The claimed asymptotic for $t=t(m)$ is equivalent to $m=(1+O(t^{-1/2}))(t^3/3)^{1/2}$.  Set $$M=M(t):=(1+Ct^{-1/2})(t^3/3)^{1/2},$$ for a universal constant $C>0$ to be specified later; we will establish the claim by showing that $\sum_{i=1}^t s(i)< s(M)$.  Notice that for $1\leq i\leq t$, superadditivity gives
$$x(M) \geq \lfloor M/i \rfloor \cdot x(i)=(1+O(t^{-1/2})) \frac{M}{i} \cdot x(i)$$
(using $M/i \gg t^{1/2}$ to bound the error term).  Thus (due to the positivity of the $y(i)$'s) we have
$$\sum_{i=1}^t x(i)y(i) \leq \sum_{i=1}^t (1+O(t^{-1/2})) \frac{i}{M} \cdot x(M)y(i)=(1+O(t^{-1/2})) M^{-1} x(M) \sum_{i=1}^t iy(i).$$
We can write the last sum as
\begin{equation}\label{eq:partition}
\underbrace{y(1)}_{1\text{ term}}+\underbrace{y(2)+y(2)}_{2\text{ terms}}+\cdots+\underbrace{y(t)+\cdots+y(t)}_{t\text{ terms}}.
\end{equation}
With this in mind, consider the sequence with $1$ term equal to $1$, then $2$ terms equal to $2$, and so on, up to $t$ terms equal to $t$.  Using a greedy algorithm, we partition this sequence into several subsequences each with sum between $M-t+1$ and $M$, and one final subsequence with sum at most $M$; this partition has at most
$$\frac{\sum_{i=1}^t i^2}{M-t+1} + 1=(1+O(t^{-1/2})) \frac{t^3/3}{M}$$
parts.  For each part of the partition, the  corresponding terms in \eqref{eq:partition} sum to at most $y(M)$ by superadditivity.  In total, we obtain
$$\sum_{i=1}^t x(i) y(i) \leq (1+O(t^{-1/2}))M^{-1}x(M) \cdot (1+O(t^{-1/2})) \frac{t^3/3}{M} \cdot y(M)=\frac{1+O(t^{-1/2})}{(1+Ct^{-1/2})^2}x(M)y(M).$$
The right-hand side is at most $x(M)y(M)$ as long as $C$ is sufficiently large.  This establishes the claim and completes the proof of the lemma.
\end{proof} 

Applying \cref{lem:key} to the subsequences extracted in \cref{subsec:involutions} (with $m:=\min(r,r')$) completes the proof of \cref{thm:main} in the special case where the entries of $\ba$ and $\bb$ are real. 

\begin{remark}
    The proof of \cref{lem:key}  can be generalized to show that for if $\{x_i(1), \ldots, x_i(m)\}_{1 \leq i \leq d}$ are superadditive sequences of positive reals, then the numbers $s(i):=x_1(i)x_2(i)\cdots x_d(i)$ (for $1 \leq i \leq m$) has $\Omega_d(m^{d+1})$ distinct subset sums.
It is not clear to us if this has an interpretation in terms of the minimum possible size of the set of permutation-twisted dot products
    $$S(\mathbf{a}_1, \dots, \mathbf{a}_d) = \left\{\sum_{i=1}^n a_{1, \pi_1(i)} a_{2, \pi_2(i)}\cdots a_{d, \pi_d(i)} : \pi_1, \pi_2, \dots, \pi_d\in \mathfrak S_n\right\}$$
    when each $\mathbf{a}_i = (a_{i, 1}, \dots, a_{i, n})$ has distinct coordinates. Nevertheless, we conjecture that, generalizing our result when $d=2$, such a set always has size $\Omega_d(n^{d+1})$. 
\end{remark}  

\section{Complex Numbers} 
In this section, we prove \cref{thm:gamma} and use it to establish \cref{thm:main} when $\mathbb K=\mathbb C$. 

\begin{proof}[Proof of \cref{thm:gamma}]
Let $g_d(m)$ be the minimum possible number of distinct subset sums of an $m$-element subset of $\mathbb R^d$ with no $d$-element linearly dependent subsets.  We proceed by induction on $d$.  For $d=1$, a simple inductive argument (as in, e.g., the proof of \cite[Theorem 3]{N95}) shows that every set of $r$ real numbers of the same sign has at least $r(r+1)/2+1$ distinct subset sums. Among any set of nonzero real numbers, at least half have the same sign, so $g_1(m)\geq m^2/4$. 

We now describe the induction step.  Let $d \geq 2$, and assume that there exists $\gamma_{d-1}>0$ such that $g_{d-1}(m)\geq\gamma_{d-1}m^d$ for all $m$. 
We will show that 
\begin{equation*}\label{eq:g_recurrence} 
g_d(m)\geq g_{d}(m-1)+g_{d-1}(\left\lceil(m-1)/2\right\rceil)
\end{equation*} 
for all $m\geq d$, which will complete the proof. 

Fix $m\geq d$, and let $A\subseteq \mathbb R^d$ be an $m$-element set with no $d$-element linearly dependent subsets. The condition $m\geq d$ ensures that no element of $A$ is a scalar multiple of any other, as otherwise any $d$-element subset containing both would be linearly dependent. For each $w\in A$, let $\mathcal A_w$ be the collection of $2$-element sets $\{u,v\}\subseteq A\setminus\{w\}$ such that $v-u$ is a scalar multiple of $w$. Note that no pair $\{u,v\}$ can belong to $\mathcal A_w$ and $\mathcal A_{w'}$ for distinct $w,w'$ (since $w$ and $w'$ are linearly independent). Hence, there is some $w^*\in A$ such that $|\mathcal A_{w^*}|\leq \binom{m}{2}/m=(m-1)/2$. 

Let $H$ be orthogonal complement of $w^*$, and let $A'\subseteq H$ be the orthogonal projection of $A\setminus\{w^*\}$ onto $H$. The choice of $w^*$ guarantees that $|A'|\geq m-1-|\mathcal A_{w^*}|\geq (m-1)/2$. Because no $d$-element subset of $A$ is linearly dependent, we know that no $(d-1)$-element subset of $A'$ is linearly dependent. Thus, the induction hypothesis gives that $A'$ has at least $g_{d-1}(\left\lceil(m-1)/2\right\rceil)$ distinct subset sums. For each subset sum $u'$ of $A'$, there is a subset sum $u$ of $A\setminus\{w^*\}$ such that $u'$ is the orthogonal projection of $u$ onto $H$ and such that $u+w^*$ is not a subset sum of $A\setminus\{w^*\}$. It follows that the number of distinct subset sums of $A$ is at least $g_{d}(m-1)+g_{d-1}(\left\lceil(m-1)/2\right\rceil)$, as desired. 
\end{proof}

The final ingredient is Beck's Theorem~\cite{B83} (a consequence of the Szemer\'edi--Trotter Theorem) from discrete geometry. 

\begin{theorem}[\cite{B83}] \label{thm:beck}
There is a constant $C>0$ such that for every set $X \subseteq \mathbb{R}^2$ of $n$ points, one of the following holds: 
\begin{itemize}
\item Some line contains at least $n/C$ points in $X$. 
\item There is a collection of $n^2/C$ lines each containing at least two points of $X$. 
\end{itemize} 
\end{theorem} 

We can now complete the proof of our main result over complex numbers. 

Let $\ba=(a_1,\ldots,a_n)$ and $\bb=(b_1,\ldots,b_n)$ be $n$-tuples of complex numbers such $a_1,\ldots,a_n$ are distinct and $b_1,\ldots,b_n$ are distinct.  Identifying $\ba,\bb$ with subsets of $\mathbb{R}^2$ and applying Beck's Theorem leads to the following two cases (with the constant $C$ from \Cref{thm:beck}).

\medskip 
\noindent {\bf Case 1.} Suppose that there are collinear subsets $\ba' \subseteq \ba$ and $\bb' \subseteq \bb$ each of size $n/C$.  Without loss of generality, we may assume that $\ba'=(a_1, \ldots, a_{n/C})$ and $\bb'=(b_1, \ldots, b_{n/C})$.  The size of the set $S(\ba,\bb)$ remains unchanged if we shift all of the elements of $\ba$ by a constant or multiply all of the elements by a fixed nonzero complex number, so without loss of generality we may assume that the elements of $\ba'$ are in fact all real; likewise we may assume that the elements of $\bb'$ are all real.  Now consider the set $\mathfrak{S}'_n$ of permutations of $[n]$ that permute the elements in $[n/C]$ and fix all of the elements $n/C+1, \ldots, n$.  Applying the real case of \cref{thm:main} (already proven in \cref{sec:reals}), we see that the permutations $\pi \in \mathfrak{S}'_n$ already witness $\Omega(n^3)$ distinct values of $S(\ba,\bb; \pi)$.  


\medskip 
\noindent {\bf Case 2.} Suppose that there is a set $\mathcal{L}$ of $n^2/C$ lines each containing at least two points of $\ba$ (the remaining case where there are $n^2/C$ lines each containing at least two points of $\bb$ is identical).  By the popularity principle, there are $n/2C$ points each contained in at least $n/2C$ lines of $\mathcal{L}$; without loss of generality, we may assume that these points are $a_1, \ldots, a_{n/2C}$.  We claim that it suffices to find disjoint pairs of indices $(j_1,k_1), \ldots, (j_{n/6C}, k_{n/6C})$ such that none of the $n/6C$ complex numbers
\begin{equation}\label{eq:ratios}
(a_{k_1}-a_{j_1})(b_2-b_1), \ (a_{k_2}-a_{j_2})(b_4-b_3), \ \ldots, \ (a_{k_{n/6C}}-a_{j_{n/6C}})(b_{n/3C}-b_{n/3C-1})
\end{equation}
is a real multiple of another.  Indeed, the argument from the beginning of \Cref{sec:reals} shows that $S(\ba,\bb)$ contains a translate of the set of subset sums of the numbers in \eqref{eq:ratios}, and \Cref{thm:gamma} guarantees that these numbers have at least $\Omega(n^3)$ subset sums. 

We obtain the pairs $(j_i,k_i)$ iteratively.  To start, set $(j_1,k_1):=(1,2)$.  Suppose that we have already obtained the disjoint pairs $(j_1, k_1), \ldots, (j_s,k_s)$ for some $s<N/6C$.  Let $j_{s+1}$ be the smallest element of $[n/2C]$ that has not yet appeared in any of the pairs.  Among the $N/2C$ lines of $\mathcal{L}$ containing $a_{j_{s+1}}$, at most $s$ point in directions parallel to the quantities $$\frac{(a_{k_i}-a_{j_i})(b_{2i}-b_{2i-1})}{b_{2s+2}-b_{2s+1}}$$
for $1 \leq i \leq s$, and at most $2s$ contain points that have already appeared in the earlier pairs $(a_{j_i}, a_{k_i})$ for $1 \leq i \leq s$; these two possibilities together account for at most $3s \leq N/2C-3$ lines.  Choose any remaining line of $\mathcal{L}$ containing $a_{j_{s+1}}$, and take $k_{s+1}$ so that $a_{k_{s+1}}$ is any other point on this line.  This choice guarantees that $(a_{k_{s+1}}-a_{j_{s+1}})(b_{2s+2}-b_{2s+1})$ is not a real multiple of $(a_{k_i}-a_{j_i})(b_{2i}-b_{2i-1})$ for any $1 \leq i \leq s$.  The procedure produces $N/6C$ pairs, as desired.  This completes the proof.

\section{Fields of characteristic $0$} 
We have shown that there is a constant $c>0$ such that $|\mathcal S(\ba,\bb)|\geq cn^3$ for all $\ba,\bb\in\mathbb C^n$ with distinct entries. We claim that the full version of \cref{thm:main} holds with the same value of $c$.   

Let $\mathbb K$ be a field of characteristic $0$, and let $\ba=(a_1,\ldots,a_n)$ and $\bb=(b_1,\ldots,b_n)$ be tuples in $\mathbb K^n$ each with distinct entries. Let $\mathbb K'$ be the field extension of $\mathbb Q$ generated by $a_1,\ldots,a_n,b_1,\ldots,b_n$. It is well known that every finitely generated field extension of $\mathbb Q$ embeds into $\mathbb C$. Thus, there exists a field embedding $\iota\colon\mathbb K'\to\mathbb C$. Consider the tuples $\iota(\ba):=(\iota(a_1),\ldots,\iota(a_n))$ and $\iota(\bb):=(\iota(b_1),\ldots,\iota(b_n))$ in $\mathbb{C}^n$. Then $|\mathcal S(\ba,\bb)|=|\mathcal S(\iota(\ba),\iota(\bb))|\geq cn^3$ by the complex case of \Cref{thm:main}. 

\section*{Acknowledgments}
Colin Defant was supported by a Benjamin Peirce Fellowship at Harvard University.  Noah Kravitz was supported in part by a NSF Mathematical Sciences Postdoctoral Research Fellowship under grant DMS-2501336.  We thank Roger Lid\'on Ardanuy for posing this question a few years ago, and Daniel Zhu for helpful discussions.

\end{document}